\documentclass{amsart}

\usepackage{amsmath,amssymb,amsthm}

\usepackage{pinlabel}

\usepackage[nohug,heads=LaTeX]{diagrams}

\newtheorem{prop}{Proposition}[section]
\newtheorem{thm}[prop]{Theorem}
\newtheorem{lem}[prop]{Lemma}

\theoremstyle{definition}

\newtheorem{rem}[prop]{Remark}

\newtheorem*{ack}{Acknowledgement}

\def\co{\colon\thinspace}

\newcommand{\C}{\mathbb C}
\newcommand{\CP}{\mathbb{C}\mathrm{P}}

\newcommand{\rmd}{\mathrm d}

\newcommand{\rme}{\mathrm e}

\newcommand{\rmi}{\mathrm i}

\newcommand{\wtK}{\widetilde{K}}

\newcommand{\wtL}{\widetilde{L}}
\newcommand{\loc}{\mathrm{loc}}

\newcommand{\MM}{\mathcal M}
\newcommand{\wtM}{\widetilde{M}}

\newcommand{\NN}{\mathcal N}

\newcommand{\bfp}{\mathbf p}

\newcommand{\bfq}{\mathbf q}
\newcommand{\wtQ}{\widetilde{Q}}

\newcommand{\R}{\mathbb R}
\newcommand{\RP}{\mathbb{R}\mathrm{P}}

\newcommand{\calS}{\mathcal{S}}

\newcommand{\wtW}{\widetilde{W}}
\newcommand{\WZ}{W_{\!Z}}

\newcommand{\wtX}{\widetilde{X}}

\newcommand{\Z}{\mathbb Z}
\newcommand{\wtZ}{\widetilde{Z}}

\DeclareMathOperator{\ev}{\mathrm{ev}}

\DeclareMathOperator{\id}{\mathrm{id}}
\DeclareMathOperator{\Int}{\mathrm{Int}}

\DeclareMathOperator{\Wh}{\mathrm{Wh}}

\begin{document}

\author[H. Geiges]{Hansj\"org Geiges}
\address{Mathematisches Institut, Universit\"at zu K\"oln,
Weyertal 86--90, 50931 K\"oln, Germany}
\email{geiges@math.uni-koeln.de}
\author[M. Kwon]{Myeonggi Kwon}
\address{Department of Mathematics Education, Sunchon National University,
Suncheon 57992, South Korea}
\email{mkwon@scnu.ac.kr}
\author[K. Zehmisch]{Kai Zehmisch}
\address{Fakult\"at f\"ur Mathematik, Ruhr-Universit\"at Bochum,
Universit\"atsstra{\ss}e 150, 44780 Bochum, Germany}
\email{kai.zehmisch@rub.de}

\title[Fillings of unit cotangent bundles]{Diffeomorphism type of
symplectic fillings of unit cotangent bundles}

\date{}

\begin{abstract}
We prove uniqueness, up to diffeomorphism, of symplectically aspherical
fillings of certain unit cotangent bundles, including those
of higher-dimensional tori.
\end{abstract}

\thanks{This research is part of a project in the SFB/TRR 191
\textit{Symplectic Structures in Geometry, Algebra and Dynamics}, 
funded by the DFG (Projektnummer 281071066 -- TRR 191)}

\maketitle


\section{Introduction\label{sec:intro}}
We consider the cotangent bundle $T^*L$
of a closed, connected manifold $L$ with its
Liouville form $\bfp\,\rmd\bfq$, which induces a contact
form on the unit cotangent bundle $ST^*L$ (with respect to any
Riemannian metric on~$L$). In what follows, it will always be understood
that $ST^*L$ is endowed with this natural contact structure.
The unit disc bundle $(DT^*L, \rmd\bfp\wedge\rmd\bfq)$, equipped with
its canonical symplectic form, is a strong symplectic (in fact, Stein)
filling of $ST^*L$. For basic notions of different types
of fillings see~\cite[Chapter~5]{geig08}.

It is a fundamental question in symplectic topology
in how many ways the contact manifold $ST^*L$
can be written as the boundary of a symplectic filling $(W,\omega)$.
For $ST^*S^2=\RP^3$ it was shown by McDuff~\cite{mcdu90}
that the unit disc bundle is the unique symplectically aspherical
(or minimal) strong filling up to diffeomorphism, and up to symplectomorphism
if one fixes the cohomology class of the symplectic form. For Stein fillings
of $ST^*S^2$, uniqueness up to Stein homotopy was established by
Hind~\cite{hind00}. Stipsicz~\cite{stip02} showed that
all Stein fillings of $ST^*T^2=T^3$ are homeomorphic to
the unit disc bundle $DT^*T^2=D^2\times T^2$.
Wendl~\cite{wend10} strengthened that last result to
uniqueness up to diffeomorphism. In fact, he proved uniqueness up
to deformation equivalence for all minimal symplectic fillings of~$T^3$.
Regarding higher genus surfaces, Sivek and Van Horn-Morris~\cite{siva17}
showed that every Stein filling of $ST^*\Sigma_g$, $g\geq 2$, is
$s$-cobordant rel boundary to the unit disc bundle $DT^*\Sigma_g$;
see also~\cite{lmy17}.

These $4$-dimensional results are typically based on techniques not
available in higher dimensions, such as foliations by holomorphic curves.

The question about the analogue in higher dimensions of the
result for $T^2$ was posed to
us by Otto van Koert and Andr\'as Stipsicz, and we answer it in the
following theorem.  From now on, `filling' without any further
specification always means strong symplectic filling.

\begin{thm}
\label{thm:torus}
The diffeomorphism type of symplectically aspherical fillings
of $ST^*T^n$ is unique.
\end{thm}

\begin{rem}
While the present paper was in preparation,
Theorem~\ref{thm:torus} has independently been obtained by
Bowden--Gironella--Moreno
\cite{bgm19} as an application of their extensive study of
Bourgeois contact structures, which they use to replace some
parts of Wendl's argument and then combine with results
from~\cite{bgz}. Our proof, by contrast, is based
directly on a refinement of the techniques developed in~\cite{bgz}.
As we shall explain, our approach establishes uniqueness of fillings of
unit cotangent bundles for a considerably larger class of base manifolds.
In the latest version of their paper, Bowden--Gironella--Moreno
more closely follow our approach and
also use holomorphic spheres rather than punctured holomorphic discs;
see~\cite[Remark~3]{bgm19}.
\end{rem}

The earliest results about the diffeomorphism type of symplectic fillings
in higher dimensions (after Gromov's work~\cite{grom85} in dimension four)
are due to Eliashberg--Floer--McDuff \cite{mcdu91}.
By treating evaluation maps on moduli spaces of holomorphic
curves with methods from algebraic topology, they
proved that every symplectically aspherical filling of the sphere $S^{2n-1}$
equipped with the standard contact structure is diffeomorphic
to the disc~$D^{2n}$. Uniqueness of these fillings
up to symplectomorphism is known only for $n=1,2$.

Starting point for the classification of fillings is an
understanding of homological restrictions.
Oancea--Viterbo \cite{oavi12} showed that the inclusion of
a subcritically Stein fillable contact manifold into any of its 
symplectically aspherical fillings induces a surjection in homology.
Ghiggini--Niederkr\"uger--Wendl \cite{gnw16}
found obstructions on the relative homology
of semi-positive symplectic fillings
in terms of belt spheres of subcritical handles.

The degree method from \cite{bgz} systematically combines
the filling by holomorphic curves technique
with the $s$-cobordism theorem and yields uniqueness, up to diffeomorphism,
of \emph{subcritical} Stein fillings for a wide range of contact manifolds.
In \cite{kwze} it was shown that the subcriticality assumption
can be dropped if instead one requires the existence
of a complex hypersurface in the filling having a suitable intersection
behaviour. As an application, it was shown there
that for $ST^*S^{2d+1}$ such types of fillings are unique
up to diffeomorphism. Critical fillings were also studied by
Lazarev~\cite{laza}, who proved uniqueness results up to
symplectomorphism (after completion) for certain classes
of flexible fillings.

In the present paper we show that the subcriticality assumption
on the filling made in \cite{bgz} can be dropped in situations where suitable
topological information is available on the manifold that is to be filled.
Our arguments apply to unit cotangent bundles $ST^*L$ of
closed, connected manifolds $L$
that admit a Lagrangian embedding into a subcritical Stein manifold.
Assuming that the second relative homology of a symplectic filling
$(W,\omega)$ of $ST^*L$ vanishes,
we first prove homological uniqueness of the filling,
i.e.\ that $H_*(W)$ is isomorphic to $H_*(DT^*L)$.
This is the content of Theorem~\ref{thm:homologytype}, where the situation
is analysed in a slightly more general setting.
If in addition $L=Q\times S^1$ with $\chi(Q)=0$, we show in
Theorem~\ref{thm:inducedmap}
that this homology isomorphism is induced by an embedding $DT^*L\rightarrow W$.

By a result of Chekanov \cite{chek98}, the fundamental group of a
Lagrangian submanifold in a subcritical Stein manifold
contains an element of infinite order,
see Proposition~\ref{prop:chekanov}.
So it is quite natural to consider manifolds $L=Q\times S^1$
that split off a circle factor.
There is an obvious Lagrangian embedding of $L$ into the subcritical
manifold $T^*Q\times\C$ that allows filling by holomorphic curves.
Using a filling by holomorphic annuli
we show that the inclusion
$L\rightarrow W$ is surjective on $\pi_1$,
see Theorem \ref{thm:pi1surjective}.
A filling by infinitely long holomorphic strips
can be used to show that the lifted inclusion $\wtL\rightarrow\wtW$
of universal covers --- which exists when the inclusion
$ST^*L\rightarrow W$ is $\pi_1$-injective --- is surjective on $H_*$,
see Proposition \ref{prop:homsurj}.
Arguments parallel to \cite{bgz} then lead us to the main result of the
present paper.

\begin{thm}
\label{thm:main}
(a) Suppose that $Q$ is a closed manifold of dimension at least three,
with Euler characteristic $\chi(Q)=0$,
satisfying one of the following assumptions:
\begin{itemize}
\item[(i)] $Q$ is a product manifold of the form $Q=N\times F$,
where $F$ is a surface different from $S^2$ and $\R P^2$.
\item[(ii)] $Q$ is aspherical.
\end{itemize}
Then any Stein filling of $ST^*(Q\times S^1)$
is homotopy equivalent to $DT^*(Q\times S^1)$.

\vspace{2mm}

(b) If $Q$ is a product of unitary groups and spheres, including at least one
$S^1$-factor, then any symplectically aspherical
filling of $ST^*(Q\times S^1)$
is diffeomorphic to $DT^*(Q\times S^1)$.
\end{thm}

Part (a) of this theorem will be proved in
Section~\ref{subsection:homotopy-type};
part (b), in Section~\ref{subsection:difftype}.
Theorem~\ref{thm:torus} is an obvious corollary of part~(b).
Actually, we are going to prove more general statements
in Theorems~\ref{thm:Stein} and~\ref{thm:diff-type}. In the theorem above
we only listed the most obvious examples illustrating those theorems.

The main technical innovations in this paper that allow us
to go beyond the results of \cite{bgz} are the
\emph{moving complex hypersurface argument} in
Section~\ref{subsection:pi1} and the \emph{analysis of infinite
holomorphic strips} in Section~\ref{subsection:strips}.

A brief word on notation: We write $D_r$ for the \emph{closed} disc of
radius $r$ in the complex plane~$\C$, centred at~$0$. The \emph{open}
disc will be denoted by~$B_r$, that is, $B_r=\Int(D_r)$.
\section{Domains in subcritical Stein manifolds
\label{sec:dominstein}}
The homology of symplectically aspherical fillings
of contact manifolds that are subcritically Stein fillable
is unique, see \cite[Theorem 1.2]{bgz}.
As we shall see in this section, uniqueness of homology
holds also for all symplectically aspherical fillings
of contact type hypersurfaces in subcritical Stein manifolds
that are not necessarily a level set of a corresponding
plurisubharmonic function.
Examples are given by the boundaries of
Weinstein tubular neighbourhoods 
of closed Lagrangian submanifolds, which by
\cite[Proposition~3.9]{bgz} are not subcritically Stein fillable.
\subsection{The Oancea--Viterbo argument revisited}
\label{sec:OVargrev}
Let $(M_Z,\xi_Z)$ be a $(2n-1)$-dimensional
contact manifold that admits a subcritical Stein filling
$(Z,\omega_Z)$. This means that the plurisubharmonic Morse function
given by the Stein structure does not have any critical points of index~$n$.

Consider a closed, connected contact type hypersurface $(M,\xi)$
in $(Z,\omega_Z)$, disjoint from $\partial Z$. Observe
that $M$ is separating because $H_{2n-1}(Z)=0$.
Denote by $D_Z\subset Z$ the closure of the component of $Z\setminus M$
not containing $\partial Z$.

\begin{rem}
As shown in \cite[Theorem 3.4]{geze12},
the contact manifold $(M,\xi)$ is a \emph{convex} boundary of
the symplectic manifold $\big(D_Z,\omega_Z|_{D_Z}\big)$,
so the latter constitutes a symplectic filling.
Alternatively, one may appeal to~\cite[Remark~3.3]{suze17},
which gives an elementary argument. In order to apply either
reference in the present setting, one needs to appeal to
Cieliebak's splitting theorem for subcritical Stein manifolds
\cite[Section~14.4]{ciel12}.

The symplectic form $\omega_Z|_{D_Z}$ is exact, but it need not be of
Liouville type, i.e.\ there need not be a primitive $1$-form
of $\omega_Z|_{D_Z}$ that restricts to a contact form for $\xi$
on the boundary~$M$.
\end{rem}

\begin{thm}
\label{thm:homologytype}
Let $(W,\omega)$ be a symplectic filling of $(M,\xi)$
that satisfies one of the following conditions:
\begin{enumerate}
\item [(i)] $H_2(W,M)=0$;
\item [(ii)] $M$ is simply connected and $(W,\omega)$ symplectically
aspherical;
\item [(iii)] $H^1(M;\R)=0$ and $\omega$ is exact.
\end{enumerate}
Then $W$ and $D_Z$ have isomorphic homology.
\end{thm}

\begin{rem}
\label{rem:OV}
(1) Under assumption (i) or (iii) of the theorem, $(W,\omega)$
is symplectically aspherical for obvious (co-)homological
reasons.

(2) Under assumption (i) or~(ii),
Oancea--Viterbo proved in \cite{oavi12} that
the inclusion map $M\rightarrow W$ induces a surjection in homology,
cf.\ \cite[Theorem 3.2]{bgz}.

(3) It was shown in \cite[Proposition 3.5]{bgz} that the normal subgroup in
$\pi_1(W)$ generated by the image of $\pi_1(M)$ in $\pi_1(W$) is equal to
$\pi_1(W)$. This group-theoretical property by itself does not imply that the
inclusion map $M\rightarrow W$ induces a surjection of fundamental groups.
For example, the normal closure of any non-trivial subgroup
of the alternating group of degree five $A_5$, such as the Klein four-group,
is equal to the full group $A_5$, because $A_5$ is a simple group.
\end{rem}

\begin{proof}[Proof of Theorem \ref{thm:homologytype}]
We define a new symplectic manifold
\begin{equation}
\label{eqn:WZ}
(\WZ,\Omega_Z):= (W,\omega)\cup_{(M,\xi)}\big(Z\setminus\Int(D_Z),
\omega_Z\big)
\end{equation}
by replacing $D_Z$ in $Z$ with $W$.
Any of the assumptions (i)--(iii) in Theorem \ref{thm:homologytype}
implies that $(\WZ,\Omega_Z)$ is a symplectically aspherical filling of
$(M_Z,\xi_Z)$:
The argument for assumptions (i) and (ii) is given in
\cite[Remark~3.3~(2) and Lemma~3.4]{bgz};
in case (iii) a straightforward argument in de Rham theory
shows that $\Omega_Z$ is globally exact.
In fact, in this last case one can find a global primitive for $\Omega_Z$ that
restricts to a contact form for $\xi$ on~$M$, so $(M, \xi)$
is of restricted contact type in $(\WZ,\Omega_Z)$.
Therefore, in all three cases we can appeal to \cite[Theorem 1.2]{bgz}
to conclude that $H_k(\WZ)\cong H_k(Z)$ for all $k\in\Z$.
 
Denote by $\ell$ the dimension of the CW complex
obtained from $Z$ by following the negative gradient flow
of the plurisubharmonic Morse function of the Stein structure on~$Z$.
Since the Stein structure is assumed to be subcritical, we
have $\ell\leq n-1$. Then $H_k(\WZ)$ is trivial for $k\geq\ell+1$,
and a free abelian group for $k=\ell$, since there are no $n$-cells.
The Mayer--Vietoris sequence yields
\[ H_k(W)\oplus H_k(\WZ\setminus W) \cong H_k(M)\cong
H_k(D_Z)\oplus H_k(Z \setminus D_Z)\;\;\text{for $k\geq\ell+1$}. \] 
Since $\WZ\setminus W=Z \setminus D_Z$ 
it follows that $H_k(W)\cong H_k(D_Z)$ 
for $k\geq\ell+1$.

It remains to prove $H_k(W)\cong H_k(D_Z)$ for $k=0,1,\ldots,\ell$.
We first observe, by combining Poincar\'e duality and excision in cohomology,
that
\[ H_k(W)\cong H^{2n-k}(\WZ,\WZ\setminus W).\]
Secondly, since $\ell\leq n-1$, the condition $k\leq\ell$
translates into $2n-1-k\geq\ell+1$. By using the universal coefficient
theorem in the form $H^i=FH_i\oplus TH_{i-1}$, with $F,T$ denoting the
free and the torsion part, respectively, we see that the cohomology groups
of $\WZ$ vanish in degree $2n-1-k$ and $2n-k$.
Hence, the connecting homomorphism
\[ H^{2n-1-k}(\WZ\setminus W)\longrightarrow
H^{2n-k}(\WZ,\WZ\setminus W) \]
of the cohomology long exact sequence for the pair $(\WZ,\WZ\setminus W)$
is an isomorphism.

Combining these two observations, we have
\[ H_k(W)\cong H^{2n-1-k}(\WZ\setminus W)\;\;\text{for $k=0,1,\ldots,\ell$.}\]
The same argument yields $H_k(D_Z)\cong H^{2n-1-k}(Z \setminus D_Z)$
for $k=0,1,\ldots,\ell$.
Since $\WZ\setminus W$ and $Z\setminus D_Z$ coincide,
we conclude that $H_k(W)\cong H_k(D_Z)$ for $k=0,1,\ldots,\ell$.
\end{proof}
\subsection{Weinstein neighbourhoods}
\label{sec:weineids}
Examples to which Theorem \ref{thm:homologytype} applies
can be obtained by intrinsic Weinstein surgery.
But the most prominent applications of Theorem \ref{thm:homologytype}
are Weinstein tubular neighbourhoods of closed Lagrangian
submanifolds $L\subset Z$ in a subcritical Stein manifold~$Z$.
The contact type hypersurface in question
is the unit cotangent bundle $M=ST^*L$ for some metric on~$L$,
with corresponding domain $D_Z=DT^*L$.
In this situation Theorem \ref{thm:homologytype}~(i) implies
that the homology type of any symplectic filling $W$ of $M$
is the one of $L$, provided $H_2(W,M)=0$.
Notice that for $\dim L=n\geq 3$ we have $H_2(DT^*L,ST^*L)=0$.
The following proposition says that such fillings $W$ do not
exist for~$n=2$.

\begin{prop}
\label{prop:geq3}
For $L$ a closed (possibly non-orientable) surface and
$M=ST^*L\subset Z$ as described, there is no symplectic filling $W$
of $M$ with $H_2(W,M)=0$.
\end{prop}

\begin{proof}
Suppose $M$ admitted a symplectic filling $W$ with $H_2(W,M)=0$.
Then, by Theorem~\ref{thm:homologytype}, we have
\begin{eqnarray*}
0\;=\;H_2(W,M) & \cong & H^2(W)\\
               & \cong & FH_2(W)\oplus TH_1(W)\\
               & \cong & FH_2(L)\oplus TH_1(L).
\end{eqnarray*}
The condition $FH_2(L)=0$ would force $L$ to be non-orientable,
but then $TH_1(L)=\Z_2$, so this is not possible.
\end{proof}

Concerning cases (ii) and (iii) of Theorem~\ref{thm:homologytype},
the following proposition says that these
are irrelevant for the filling of unit cotangent bundles. In particular,
there are specific topological restrictions on manifolds that can be
realised as Lagrangian submanifolds in a subcritical Stein manifold.

\begin{prop}
\label{prop:chekanov}
Conditions (ii) and (iii) in Theorem \ref{thm:homologytype}
are never satisfied for $M=ST^*L$.
\end{prop}

\begin{proof}
By a result of Chekanov \cite{chek98}
there exists a non-constant holomorphic disc $\Delta$
in $(Z,\omega_Z)$ with boundary $\partial\Delta$ on $L$. The reason is the
following. The Lagrangian submanifold $L$ is 
displaceable in the completion $\hat{Z}$ of $Z$, since
$\hat{Z}$ symplectically splits off a $\C$-factor by Cieliebak's
splitting theorem, see \cite[Section~14.4]{ciel12}.
There are no non-constant holomorphic spheres in~$\hat{Z}$ by exactness of
the Stein symplectic form. If there were no non-constant holomorphic disc
in $(Z,\omega_Z)$ with boundary on $L$, the energy bound
on non-displacing symplectomorphisms 
in the main theorem of \cite{chek98} would be infinite,
i.e.\ $L$ would not be displaceable under any symplectomorphism.

Let $\lambda$ be a primitive $1$-form for the symplectic form~$\omega_Z$.
Then the $1$-form $\lambda_L:=\lambda|_{TL}$ is closed
($L$ being Lagrangian), and it represents a non-trivial class in
$H^1(L;\R)$, since it integrates non-trivially over $\partial\Delta$
by the theorem of Stokes. Thus, condition (iii) is violated.

Further, the circle $\partial\Delta\subset L$ represents an element of
infinite order in $\pi_1(L)$, and this lifts to an element of
infinite order in~$\pi_1(M)$, which violates condition~(ii).
\end{proof}
\subsection{The split situation}
\label{subsection:split}
The isomorphism between the homology of $D_Z$ and that of $W$
in Theorem~\ref{thm:homologytype} is in some sense
natural on a formal algebraic level, but it is not, in general,
induced by a map between these two manifolds. In this and the
following section we describe a situation where it is.

Let $Q$ be a closed, connected $(n-1)$-dimensional manifold.
The manifold $L=Q\times S^1$ embeds as a Lagrangian submanifold
into $T^*Q\times\C=:(Z,\omega_Z)$, by first embedding $L$ as the zero
section into $T^*L=T^*Q\times T^*S^1$, followed by the embedding of
the factor $T^*S^1$ into $\C$ as a neighbourhood of the
unit circles $S^1\subset\C$. In $T^*L\subset Z$ we consider the
contact type hypersurface $M=ST^*L=\partial (DT^*L)$, and we form the
symplectic manifold $(\WZ,\Omega_Z)$ as in (\ref{eqn:WZ}) by replacing
$DT^*L=:D_Z$ with a given symplectic filling $(W,\omega)$ of~$M$.
This is illustrated in Figure~\ref{figure:split1}, where the lightly
shaded solid torus represents $DT^*L$, which is being replaced by~$W$.
Beware that the latter no longer has a rotational symmetry, in general,
hence the note of warning in Figure~\ref{figure:split1}.

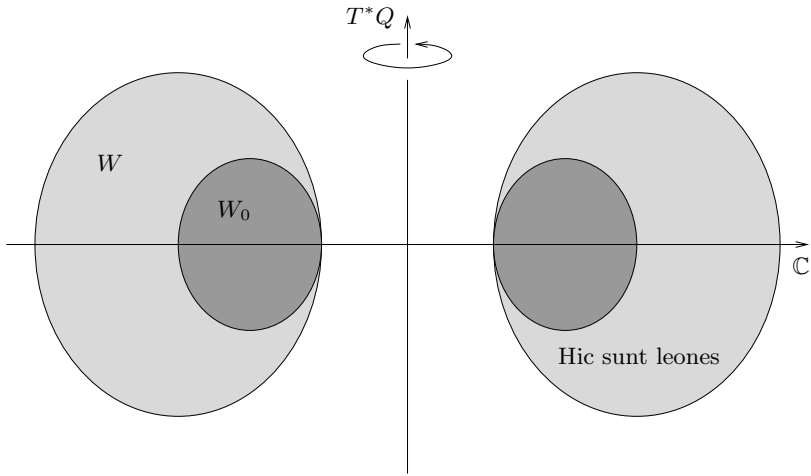
\begin{figure}[h]
\labellist
\small\hair 2pt
\pinlabel $W$ at 66 196
\pinlabel $W_0$ at 144 166
\pinlabel $\text{Hic sunt leones}$ [bl] at 345 67
\pinlabel $\C$ [t] at 500 140
\pinlabel $T^*Q$ [r] at 248 287
\endlabellist
\centering
\includegraphics[scale=0.6]{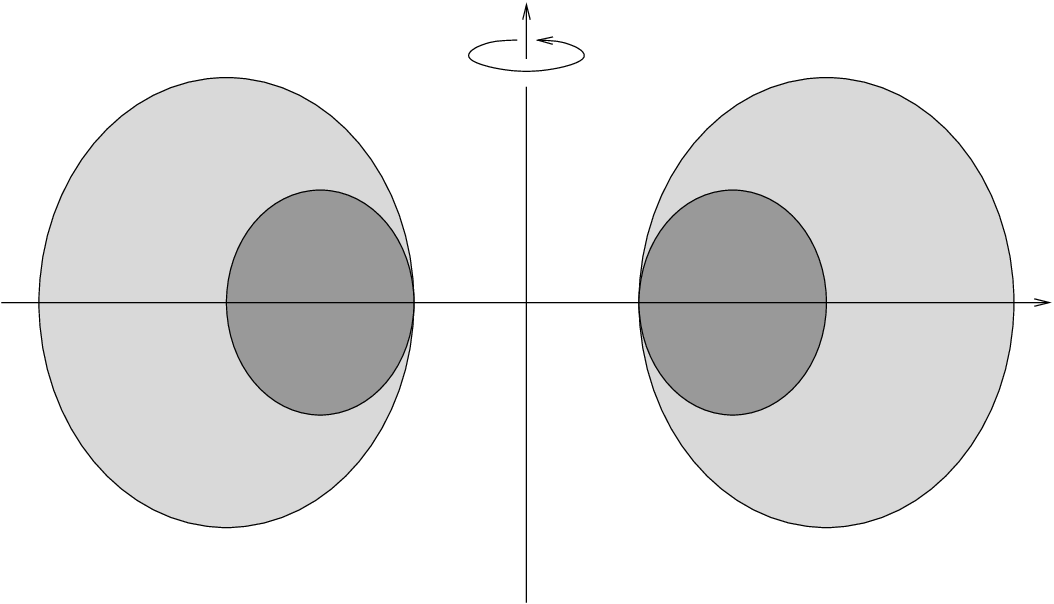}
  \caption{The setting of Theorem~\ref{thm:inducedmap}.}
  \label{figure:split1}
\end{figure}

Thanks to the splitting $L=Q\times S^1$, we have global
sections of $T^*L$ of the form $Q\times S^1\times\{t\}$,
where $Q$ is identified with the zero section in
$T^*Q$, and $S^1\times\{t\}\subset S^1\times\R=T^*S^1$.
This allows us to find inside $DT^*L$ a smaller copy $W_0$ of itself
which intersects $ST^*L$ in a copy of $Q\times S^1$. This copy
$W_0$ may be assumed to sit inside a thin tubular neighbourhood
of $ST^*L\subset DT^*L$, and hence can also be found inside
a collar of the manifold $W$ replacing $DT^*L$,
as indicated in Figure~\ref{figure:split1}.
\subsection{A homology equivalence in the split situation}
\label{subsection:homeq}
For better reference, we formulate the main hypothesis of the
following theorem separately. We say that \emph{hypothesis {\rm (H)}
is satisfied} if one of the following conditions holds, where
we recall that $M=ST^*(Q\times S^1)$.

\begin{itemize}
\item[(H-i)] $H_2(W,M)=0$ and $\chi(Q)=0$;
\item[(H-ii)] $Q=S^1\times N$ with $N$ a manifold of
dimension at least~$1$, and $(W,\omega)$ is symplectically aspherical.
\end{itemize}

\begin{rem}
Under condition (H-i), $\dim Q\geq 2$ follows from Proposition~\ref{prop:geq3}.
\end{rem}

\begin{thm}
\label{thm:inducedmap}
In the situation as just described, and under the assumption that hypothesis 
{\rm (H)} is satisfied,
the embedding $W_0\rightarrow W$ induces isomorphisms on homology.
\end{thm}

\begin{proof}
Up to homotopy, it may be assumed that $W_0$ touches the boundary
of $W$ from the inside along a tubular neighbourhood of $Q\times S^1
\subset ST^*L$. This $Q\times S^1$ bounds $Q\times D^2$ in $Q\times\C$.
By gluing in a smaller and a larger thickening of this $Q\times D^2$
inside $T^*Q\times\C$ we obtain the manifolds $W_0'$ and $W'$ as
shown in Figure~\ref{figure:thickening1}. The manifold
$W_1'$ is obtained as a further thickening of $W'$ in the
radial direction of $\C$ and the fibres of~$T^*Q$.

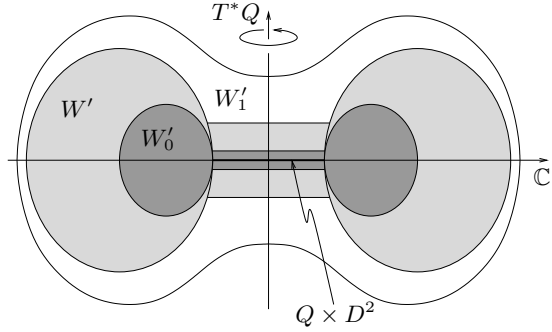
\begin{figure}[h]
\labellist
\small\hair 2pt
\pinlabel $W'$ at 68 196
\pinlabel $W_0'$ at 146 166
\pinlabel $W_1'$ at 217 205
\pinlabel $\C$ [t] at 517 142
\pinlabel $T^*Q$ [r] at 250 287
\pinlabel ${Q\times D^2}$ [t] at 318 13
\endlabellist
\centering
\includegraphics[scale=0.4]{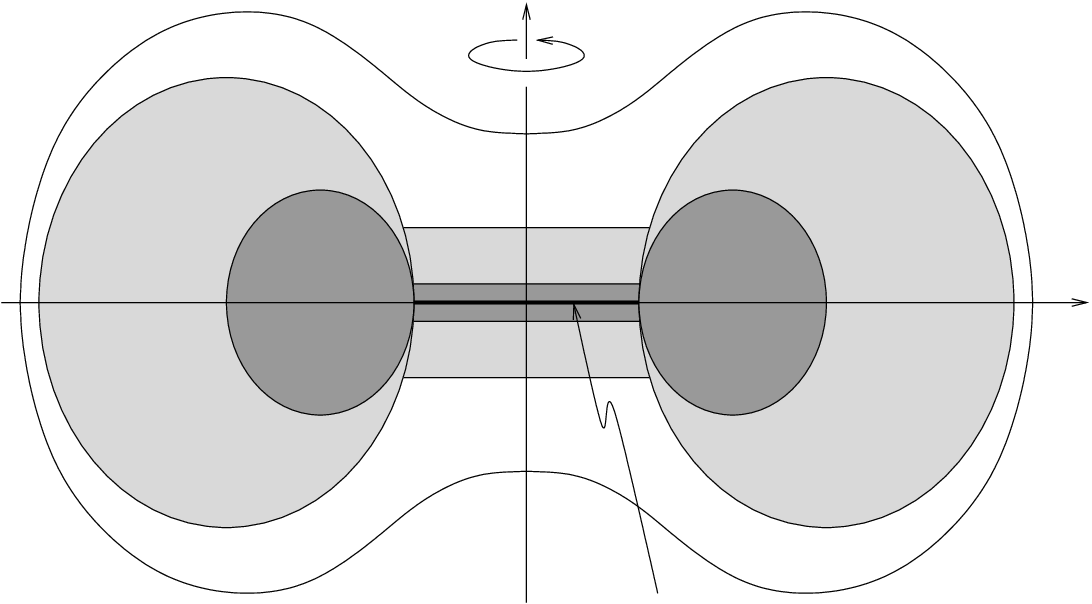}
  \caption{Proof of Theorem~\ref{thm:inducedmap} --- the thickenings.}
  \label{figure:thickening1}
\end{figure}

Thanks to $\chi(Q)=0$ (or even $Q=S^1\times N$)
we have a nowhere vanishing
section of $T^*Q$, which allows us to isotope $W_0'$ inside
$W_1'$ to a position as shown in Figure~\ref{figure:thickening2},
by first shrinking it in the $\C$-direction such that it becomes
positioned inside the `neck' formed by $DT^*(Q\times S^1)$
in $T^*Q\times\C$.

\begin{figure}[h]
\labellist
\small\hair 2pt
\pinlabel $W'$ at 100 187
\pinlabel $W_0'$ at 535 164
\pinlabel $\C$ [t] at 790 142
\pinlabel $T^*Q$ [r] at 241 283
\pinlabel $W_1'$ at 520 241
\endlabellist
\centering
\includegraphics[scale=0.38]{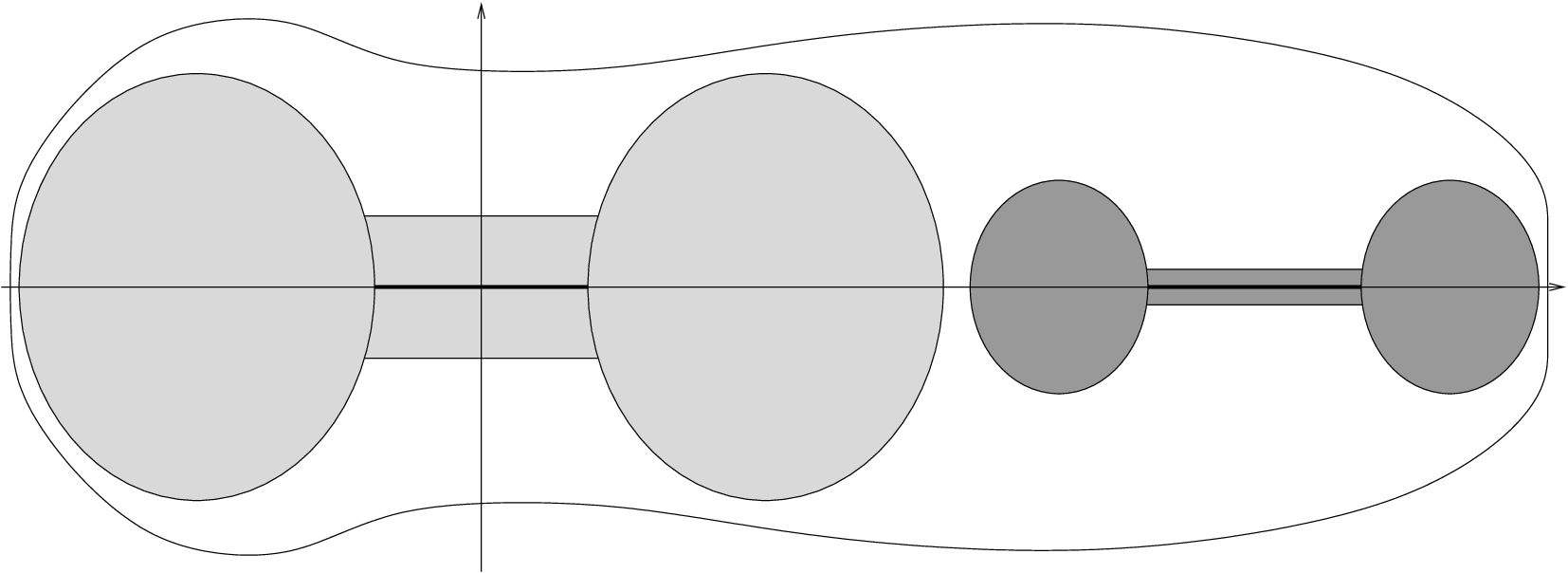}
  \caption{Proof of Theorem~\ref{thm:inducedmap} --- after the isotopy.}
  \label{figure:thickening2}
\end{figure}
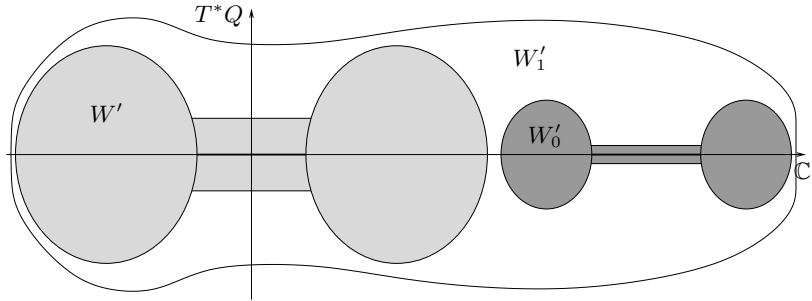

Then situation in Figure~\ref{figure:thickening2} is the one studied
in~\cite[Theorem~1.2 and Section~5]{bgz}, and the results there
say that the inclusion $W_0'\rightarrow W_1'$ is a homology
equivalence. The same is true before the isotopy of
$W_0'$ inside $W_1'$; thus, the inclusion $W_0'\rightarrow W'$
in Figure~\ref{figure:thickening1} is likewise a homology equivalence.

\begin{rem}
Hypothesis (H) is required to guarantee that in
the relevant moduli space of holomorphic spheres there is no
bubbling off. For further details see Remark~\ref{rem:nobubbles} below. 
\end{rem}

Now, in Figure~\ref{figure:thickening1}, consider the splittings
(up to homotopy)
\[ W_0'=W_0\cup_{T^*Q\times S^1}T^*Q\times D^2\]
and
\[ W'=W\cup_{T^*Q\times S^1}T^*Q\times D^2.\]
This leads to the commutative diagram
\begin{diagram}
H_{k+1}(W_0') & \rTo & H_k(T^*Q\times S^1) & \rTo
 & H_k(W_0)\oplus H_k(T^*Q\times D^2) & \rTo & H_k(W_0')\\
\dTo_{\cong}  &      & \dTo_{\id}             &      
 & \dTo_{i_*\oplus\id}                &      & \dTo_{\cong}\\
H_{k+1}(W')   & \rTo & H_k(T^*Q\times S^1) & \rTo 
 & H_k(W)\oplus H_k(T^*Q\times D^2)   & \rTo & H_k(W')
\end{diagram}
of Mayer--Vietoris sequences (continuing horizontally on either side).
By the five-lemma, the homomorphism $i_*\co H_k(W_0)\rightarrow H_k(W)$
induced by inclusion is an isomorphism for all $k\in\Z$.
\end{proof}
\subsection{A cobordism}
\label{subsection:cobordism}
Still under the assumptions of Theorem~\ref{thm:inducedmap}, and after
pushing $W_0$ a little into the interior of~$W$,
we set $X:=W\setminus\Int(W_0)$. This
defines a cobordism $\{M_0,X,M\}$ between two copies $M_0\cong M$ of $ST^*L$.

\begin{lem}
\label{lem:cobordism}
The boundary inclusions $M_0\rightarrow X$ and $M\rightarrow X$
induce isomorphisms on homology.
\end{lem}

\begin{proof}
The excision isomorphism gives $H_k(X,M_0)\cong H_k(W,W_0)$, and the latter
relative homology groups vanish for all $k$ by Theorem~\ref{thm:inducedmap}.
The relative cohomology groups $H^k(X,M)$ then vanish by Poincar\'e duality,
and hence so do the relative homology groups $H_k(X,M)$ by the
universal coefficient theorem.
\end{proof}

The proof of \cite[Lemma 5.1]{bgz}
shows that if the inclusion $M\rightarrow W$ is $\pi_1$-isomorphic,
then so are the inclusions $M_0,M\rightarrow X$.
The bundle projection $ST^*L\rightarrow L$
is $\pi_1$-isomorphic by the homotopy long exact sequence of the
fibration.
Moreover, since this bundle has a section by our assumption that
$L$ equals $Q\times S^1$, we have an inclusion $L\rightarrow W$ that factors
through $ST^*L=M=\partial W$. Hence, the $\pi_1$-isomorphicity of
the inclusion $M\rightarrow W$
reduces to the inclusion $L\rightarrow W$ being $\pi_1$-isomorphic.
\section{Holomorphic curves analysis}
We continue with the setup as described
in Section~\ref{subsection:split} and illustrated 
in Figure~\ref{figure:split1}.
\subsection{The fundamental group}
\label{subsection:pi1}
As explained in Remark~\ref{rem:OV}, the
fundamental group of a symplectically aspherical filling $(W,\omega)$
of a given contact manifold $(M,\xi)$ cannot, in general,
be determined without some specific assumptions on $(M,\xi)$.
The following theorem deals with the situation at hand,
i.e.\ the case $(M,\xi)=ST^*L$, where $L=Q\times S^1$.

\begin{thm}
\label{thm:pi1surjective}
Let $Q$ be a closed, connected $(n-1)$-dimensional manifold,
$n\geq 3$. Set $L=Q\times S^1$.
Let $(W,\omega)$ be a symplectic filling of $(M,\xi)=ST^*L$
satisfying hypothesis {\rm (H)}.
Then the inclusion $M\rightarrow W$ induces a surjection of fundamental groups.
\end{thm}

\begin{rem}
\label{rem:pi1-abelian}
(1) For this and other statements concerned only with $\pi_1$,
one may drop the condition $\chi(Q)=0$ from hypothesis (H-i).
For all statements about higher homology groups, it will be
required.

(2) If $\pi_1(Q)$ (and hence $\pi_1(M)$) is abelian, then under the assumptions
of the theorem the inclusion $M\rightarrow W$ is even
$\pi_1$-isomorphic, since $\pi_1$-surjectivity guarantees
that $\pi_1(W)$ is likewise abelian, so the fundamental group of $M$
and $W$ coincides with the respective first homology group.
On $H_1$, the inclusion $M\rightarrow W$ is an isomorphism
by Theorem~\ref{thm:inducedmap}, since for $n\geq 3$ the first homology
group of $ST^*L$ coincides with that of~$L\simeq DT^*L=W_0$.
\end{rem}

\begin{proof}[Proof of Theorem~\ref{thm:pi1surjective}]
Equip $T^*Q$ with a compatible almost complex structure
that is cylindrical in the complement of $\Int(DT^*Q)$.
Other choices are possible; all that matters to us
is that we have an almost complex structure for which
the maximum principle holds in fibre direction.
As before, we embed $DT^*L$ into $T^*Q\times\C$, and we compactify
the $\C$-factor to $\CP^1=\C\cup\{\infty\}$ with the
Fubini--Study symplectic form and its natural complex structure.

Define a symplectic manifold
\begin{equation}
\label{eqn:hatZ}
(\hat{Z},\hat{\Omega}):=
(W,\omega)\cup_{(M,\xi)}
\bigl((T^*Q\times\CP^1)\setminus\Int(DT^*L)\bigr).
\end{equation}
Compared with the definition of $\WZ$ in~(\ref{eqn:WZ}),
this amounts to a partial compactification by
the hyperplane $H_{\infty}:=T^*Q\times\{\infty\}$.
Choose a compatible almost complex structure on $(\hat{Z},\hat{\Omega})$
that is generic on $\Int(W)$, and the product almost complex structure
coming from $T^*Q\times\CP^1$ on the complement of $\Int(W)$.

Consider the moduli space $\MM$ of holomorphic spheres
$u\co\CP^1\rightarrow\hat{Z}$ subject to the following conditions,
where we fix a real number $\varrho\gg1$:
\begin{itemize}
\setlength\itemsep{.3em}
\item[(M1)] $[u]=[v\times\CP^1]$ in $H_2(\hat{Z})$
for some $v\in T^*Q\setminus DT^*Q$;
\item[(M2)] $u(z)\in H_z:=T^*Q\times\{z\}$ for $z\in\{0,\varrho,\infty\}
\subset\CP^1$.
\end{itemize}
This places us in the situation
of \cite[Section~2]{bgz}, except for an inessential difference
in the choice of the hyperplanes in the $3$-point condition~(M2).

\begin{rem}
\label{rem:nobubbles}
The moduli space $\MM$ is an oriented $(2n-2)$-dimensional manifold,
cf.\ \cite{geze12,mcsa12} and \cite[Section 2.2]{bgz}.
If $H_2(W,M)=0$, i.e.\ under condition~(H-i), the manifold
$(\WZ,\Omega_Z)$ is symplectically aspherical
thanks to \cite[Lemma 3.4]{bgz}, where it is shown that the
gluing of a symplectically aspherical manifold with
vanishing relative $H_2$ (here: $W$) and an exact symplectic
manifold (here: $(T^*Q\times\C)\setminus\Int(DT^*L)$) is
symplectically aspherical. Then one argues exactly as in
\cite[Proposition~2.3]{bgz} to see that
there is no bubbling off.
In Lemma \ref{lem:bubbling-ii} below
we shall give an argument that excludes bubbling off
under assumption (H-ii).
\end{rem}

Therefore, the evaluation map
\[ \ev\co\MM\times\CP^1\longrightarrow\hat{Z},\;\;\;
(u,z)\longmapsto u(z)\]
is proper. Moreover, the degree of the evaluation map is~$1$,
because any $u\in\MM$ whose image $u(\CP^1)$
is disjoint from $\Int(W)$ is of the form $u(z)=(v_0,z)$, $z\in\CP^1$,
for some $v_0\in T^*Q$ by the maximum principle.

Let $Z^*$ be the space obtained by removing the complex hypersurfaces
$H_0$ and $H_{\infty}$ from $\hat{Z}$ (or by removing
$H_0$ from~$\WZ$), i.e.
\begin{equation}
\label{eqn:Z*}
Z^*=\hat{Z}\setminus(H_0\cup H_{\infty})=\WZ\setminus H_0.
\end{equation}
Observe that $Z^*$ deformation retracts onto~$W$.
By positivity of intersections, we have
$u^{-1}(H_{\infty})=\{\infty\}$ and $u^{-1}(H_0)=\{0\}$
for all $u\in\MM$. Therefore, the restriction of the
evaluation map to~$\C^*$, i.e.\ the map $\MM\times\C^*\rightarrow Z^*$,
is well defined, proper and of degree $1$.

Since the intersection number $u\bullet H_0$ equals~$1$,
positivity of intersections tells us that each
$u\in\MM$ intersects $H_0$ transversely, so that the vector space
\[ T_{(u,0)}\ev (\{\mathbf{0}\}\oplus T_0\CP^1) =
T_0u(T_0\CP^1) \]
is a complex line in $T_{u(0)}\hat{Z}$ transverse to $T_{u(0)}H_0$.
Even though $\MM$ is not compact, the spheres $u\in\MM$ that
intersect $H_0$ outside a certain compact region are standard.
Therefore, for $\varepsilon>0$ sufficiently small,
the preimage $\ev^{-1}(T^*Q\times D_{\varepsilon})$,
where $D_{\varepsilon}\subset \C$ denotes the closed
disc of radius~$\varepsilon$ centred at the origin,
is a product tubular neighbourhood
of $\ev^{-1}(H_0)=\MM\times\{0\}$ in $\MM\times\CP^1$.
It follows that the inclusion map
\[ \ev^{-1}(T^*Q\times\partial D_{\varepsilon})\longrightarrow
\MM\times\C^*\]
is a homotopy equivalence.

Now consider the following commutative diagram, where the vertical
arrows are inclusion maps:
\begin{diagram}
\ev^{-1}(T^*Q\times \partial D_{\varepsilon}) && \rTo^{\ev} &&
    T^*Q\times \partial D_{\varepsilon}\\
\dTo^{\simeq}                                   &&            &&
    \dTo_{\iota}\\
\MM\times\C^*                                   && \rTo^{\ev} &&
    Z^*.
\end{diagram}
\noindent
As a map of degree~$1$, the evaluation map at the bottom of the diagram
is surjective on~$\pi_1$, cf.\ \cite[Section~2.5]{bgz}.
It follows that $\iota$ is surjective on~$\pi_1$. 
Up to homotopy, this map $\iota$ may be regarded as the
inclusion of a section $Q\times S^1\subset ST^*(Q\times S^1)=M$
into $W\subset Z^*$, which factors through the
inclusion $M\rightarrow W$. Hence, that last inclusion must
likewise be $\pi_1$-surjective.
\end{proof}

Observe that by \cite[Lemma 2.1]{bgz}, any $u\in\MM$ with
$u(\CP^1)\cap\Int(W)=\emptyset$
is of the form $u(z)=(v_0,z)$, $z\in\CP^1$,
for some $v_0\in T^*Q$.
Therefore, in order to show that the evaluation map
$\ev\co\MM\times\CP^1\rightarrow\hat{Z}$ is proper under assumption (H-ii)
of Theorem~\ref{thm:pi1surjective} (which then completes the proof of
that theorem), it suffices to establish the following statement.

\begin{lem}
\label{lem:bubbling-ii}
Under assumption {\rm (H-ii)} of Theorem~\ref{thm:pi1surjective},
any sequence of holomorphic spheres in $\MM$
that intersect $\Int(W)$ non-trivially
has a $C^{\infty}$-converging subsequence.
\end{lem}

\begin{proof}
There exists a subsequence of the given sequence that Gromov-converges
to a stable map $(u^{\alpha})_{1\leq\alpha\leq N}$.
We need to show that $N=1$.

We may assume that $u^1\bullet H_{\infty}=1$ and
$u^j\bullet H_{\infty}=0$ for $j=2,\ldots,N$.
If in addition $u^1\bullet H_0=1$, then $u^j\bullet H_0=0$ for
$j=2,\ldots,N$, which by positivity of intersection implies
that these latter spheres are disjoint from~$H_0$, and hence
homotopic to spheres in $W$. In fact, the maximum principle
forces these spheres to be contained in~$W$.
As $(W,\omega)$ is symplectically aspherical, it does not contain
spheres of positive $\omega$-energy, which implies $N=1$.

The only other possibility is that, after reindexing the bubble spheres,
we have $u^2\bullet H_0=1$ and $u^j\bullet H_0=0$ for $j\neq2$,
see Figure~\ref{figure:twobubble}. This leads to the conclusion $N=2$
by a similar argument.

\begin{figure}[h]
\labellist
\small\hair 2pt
\pinlabel $W$ at 130 213
\pinlabel $\CP^1$ [tl] at 536 142
\pinlabel $T^*Q$ [r] at 286 282
\pinlabel $u^1$ [b] at 536 189
\pinlabel $u^1$ [b] at 44 189
\pinlabel $u^2$ [b] at 326 185
\pinlabel $H_0$ [l] at 290 25
\pinlabel $H_{\infty}$ [l] at 19 25
\endlabellist
\centering
\includegraphics[scale=0.4]{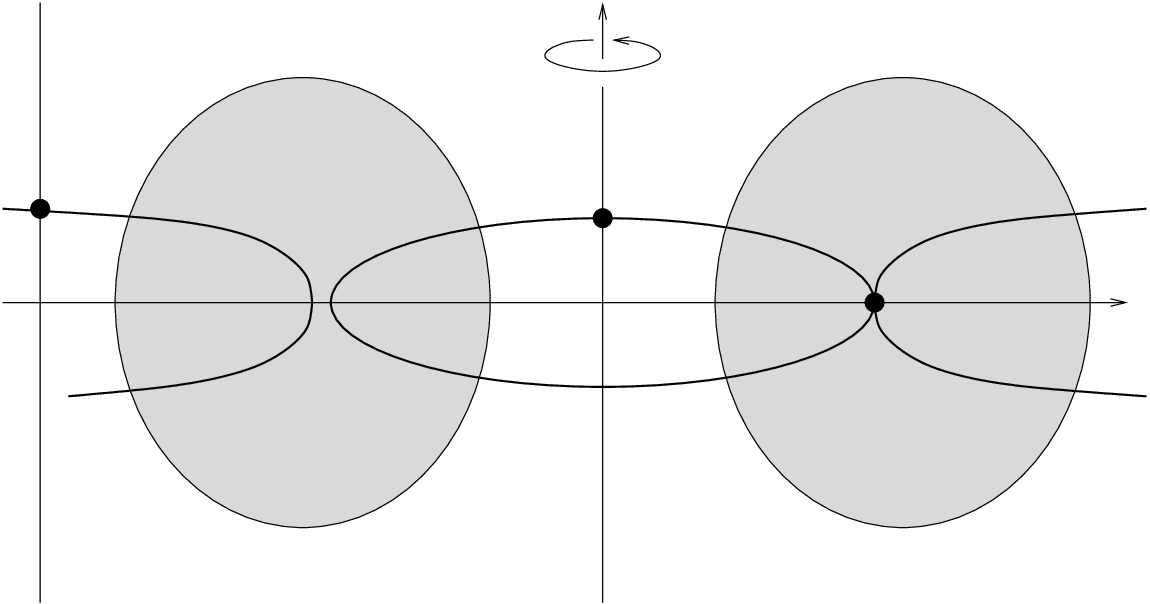}
  \caption{A potential bubbling.}
  \label{figure:twobubble}
\end{figure}
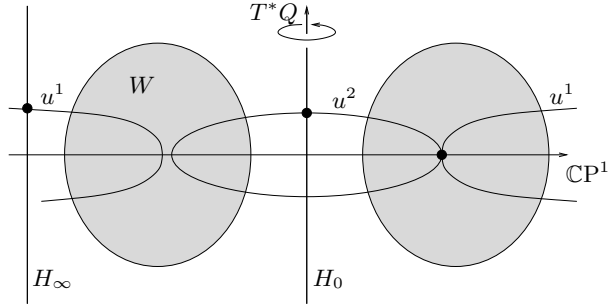

Our aim is to rule out this second case. Arguing by contradiction, we
assume that we have such a stable map $(u^1,u^2)$. We then observe
the following consequences of the maximum principle (in $T^*Q\times\C$
or separately in the factors $T^*Q$ and~$\C$). Notice that because
of $u^2\bullet H_{\infty}=0$ and positivity of intersection,
the sphere $u^2$ is disjoint from the hyperplane~$H_{\infty}$.
\begin{itemize}
\setlength\itemsep{.3em}
\item [(1)]
The sphere $u^2$ intersects $W$, since there are no non-constant
holomorphic sphere in $T^*Q\times\C$.
\item [(2)]
The sphere $u^2$ is disjoint from $(T^*Q\setminus DT^*Q)\times\CP^1$.
\item [(3)]
The sphere $u^2$ does not intersect $T^*Q\times\big(\C\setminus B_r\big)$
for $r\in(1,\infty)$ so large that $DT^*(Q\times S^1)\subset T^*Q\times
B_r$.
\end{itemize}

Therefore, and because of $u^2\bullet H_0=1$ and $u^2\bullet H_{\infty}=0$,
it is not possible to extend $(T^*Q\times [0,\infty])\setminus\Int(D_Z)$,
where $[0,\infty]\subset\RP^1\subset\CP^1$, through $W$ to a chain
with boundary $H_{\infty}-H_0$.

On the other hand, we shall now use the splitting $Q=S^1\times N$
for the construction of precisely such a chain, which is the
desired contradiction. To this end, we slightly modify the
definition of $(\hat{Z},\hat{\Omega})$ in this split setting,
without invalidating what we have said so far.

We compactify the $T^*S^1$-factor in $T^*S^1\times T^*N\times\CP^1$,
analogous to the last factor. The hypersurfaces $H_z$ are now
read as
\[ H_z=\CP^1\times T^*N\times\{z\}.\]
With this convention understood, the conclusions (1) to (3) above are
still valid.

Denote by $\MM^{\vee}$ the moduli space
of all holomorphic spheres $u\co\CP^1\rightarrow\hat{Z}$
with the following properties:
\begin{itemize}
\setlength\itemsep{.3em}
\item[(M$^{\vee}$1)] $[u]=[\CP^1\times\{v\}\times\{\infty\}]$
in $H_2(\hat{Z})$ for some $v\in T^*N\setminus DT^*N$;
\item[(M$^{\vee}$2)] $u(z)\in H_z^{\vee}:=\{z\}\times T^*N\times\CP^1$
for $z\in\{0,\varrho,\infty\}$.
\end{itemize}
The situation is illustrated in Figure~\ref{figure:hatZ}, where
the $T^*N$-factor is not shown. The shaded region is meant to
indicate a neighbourhood of $S^1\times S^1\subset\C^2$, obtained by
rotating the picture around a horizontal and a vertical axis.

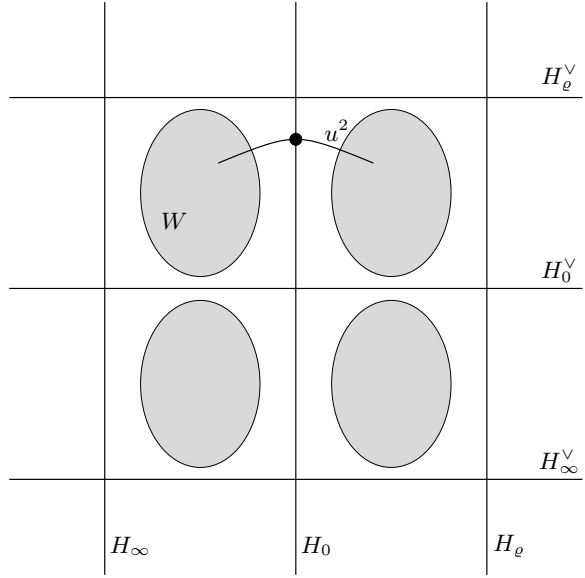
\begin{figure}[h]
\labellist
\small\hair 2pt
\pinlabel $W$ at 126 268
\pinlabel $u^2$ [bl] at 236 326
\pinlabel $H_0$ [l] at 217 23
\pinlabel $H_{\infty}$ [l] at 73 23
\pinlabel $H_{\varrho}$ [l] at 362 23
\pinlabel $H_0^{\vee}$ [b] at 416 218
\pinlabel $H_{\infty}^{\vee}$ [b] at 416 74
\pinlabel $H_{\varrho}^{\vee}$ [b] at 416 362
\endlabellist
\centering
\includegraphics[scale=0.5]{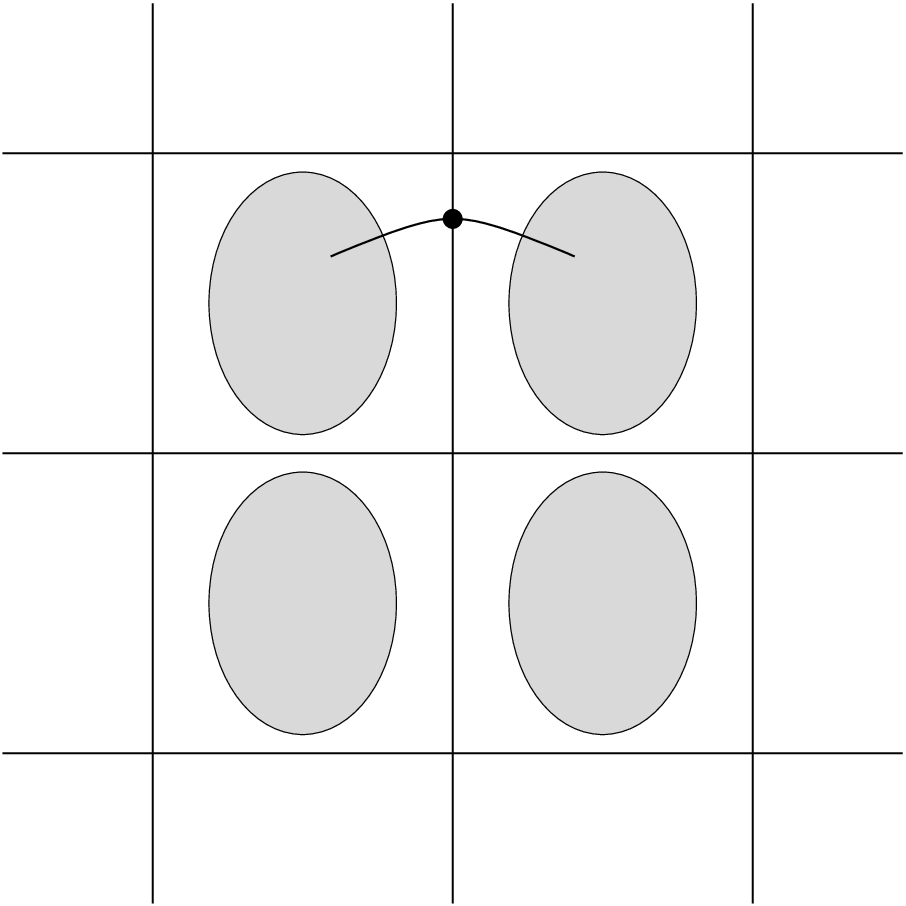}
  \caption{The space $\hat{Z}$ and the hyperplanes $H$ and $H^{\vee}$.}
  \label{figure:hatZ}
\end{figure}

\begin{rem}
\label{rem:H-foliate}
Notice that each of the hyperplanes $H_{\zeta}$, for $\zeta\in\CP^1$
sufficiently close to $0$ or~$\infty$, is foliated
by standard spheres in $\MM^{\vee}$ of the form
$z\mapsto(z,v_0,\zeta)$, $z\in\CP^1$, 
where $v_0\in T^*N$. By positivity of intersection, no
other sphere in $\MM^{\vee}$ intersects these hyperplanes.
\end{rem}

By an \emph{a priori} perturbation of the almost complex structure
over $\Int(W)$, the moduli space $\MM^{\vee}$
will be an oriented $(2n-2)$-dimensional manifold,
cf.\ \cite[Proposition~6.1]{geze12}. Here are the key arguments
to establish this fact.
All $u\in\MM^{\vee}$ intersect $H_{\infty}^{\vee}$ with
intersection number $1$, so that all $u\in\MM^{\vee}$ are simple.
Furthermore all $u\in\MM^{\vee}$ with image disjoint from $\Int(W)$
are standard spheres of the form $u(z)=(z,v_0,z_0)$, $z\in\CP^1$,
for some $v_0\in T^*Q$ and $z_0\in\CP^1$, so the freedom to choose
the almost complex structure over $\Int(W)$ suffices to achieve
regularity for all holomorphic spheres in $\MM^{\vee}$
by \cite[Remark~3.2.3]{mcsa12}.

Spheres in $\MM^{\vee}$ that do not intersect $\Int(W)$ are standard,
so to all intents and purposes this source of non-compactness can be ignored.
However, as with the original moduli space $\MM$, \emph{a priori}
it cannot be ruled out that the limit of a Gromov-convergent
sequence in $\MM^{\vee}$ happens to be a $1$-nodal stable map
$(v^1,v^2)$ representing the homology class
$[v^1]+[v^2]=[\CP^1\times\{v\}\times\{\infty\}]$;
simply replace the hyperplanes $H_0,H_{\infty}$ in the argument at
the beginning of this proof by $H_0^{\vee},H_{\infty}^{\vee}$. 
The two spheres $v^1,v^2$ can be characterised by $v^1\bullet
H_{\infty}^{\vee}=1$ and $v^2\bullet H_0^{\vee}=1$, and no intersections
with the other respective hyperplane. This intersection behaviour implies
that $v^1,v^2$ are simple, and $(v^1,v^2)$ is a simple stable map
in the sense of \cite[Definition~6.1.1]{mcsa12}.

Even so, we shall see presently that we can still use
$\MM^{\vee}$ to define a pseudocycle, which then allows us to
exclude the purported bubble sphere~$u^2$ coming from
a Gromov-convergent sequence in $\MM$. By the symmetry of the situation,
Gromov-limits of sequences in $\MM^{\vee}$,
\emph{a posteriori}, likewise lie in $\MM^{\vee}$.

Write $\calS$ for the moduli space of unparametrised $1$-nodal simple
stable spheres $(v^1,v^2)$, representing the given homology class,
with one marked point and the described
intersection behaviour with the hyperplanes $H_0^{\vee},H_{\infty}^{\vee}$,
and with non-trivial intersection of both components with $\Int(W)$.
By repeating the intersection argument for a sequence of such spheres,
we see that no further bubbling is possible. So $\calS$ is Gromov-compact
if we include the $1$-nodal spheres where the marked point
coincides with the nodal point --- these become stable maps
by introducing a ghost bubble at the nodal point \cite[p.~116]{mcsa12}.
As we shall see, these additional elements play no role
for the definition of the desired pseudocycle.

By \cite[Theorem~6.2.6]{mcsa12},
the set of regular almost complex structures, for which
$\calS$ is a smooth oriented manifold, is residual. As before,
by \cite[Remark 3.2.3]{mcsa12}
this remains true even if we require the almost complex structure
to be fixed outside $\Int(W)$.
Notice that in our situation we have an energy bound on the
holomorphic spheres, so there are only finitely many homology classes
represented by holomorphic spheres~\cite[Proposition~4.1.5]{mcsa12}.
This means that there are only finitely many splittings
$[v^1]+[v^2]$ in homology to consider. Without the energy bound there
could be countably many splittings, which would still be fine for an
application of \cite[Theorem~6.2.6]{mcsa12}, see the discussion
preceding that theorem.

Later on, we shall find further residual sets of almost complex structures
for which certain moduli spaces are manifolds, or evaluation maps
transverse to some submanifold. An almost complex structure in the
non-empty intersection of these residual sets will then settle all
transversality problems simultaneously.

Again by \cite[Theorem~6.2.6]{mcsa12}, the dimension of $\calS$ is
\begin{eqnarray*}
\dim\calS
 & = & 2n+2c_1(\CP^1)+2\cdot 1-6-2
       \cdot 1\\
 & = & 2n-2,
\end{eqnarray*}
where the first $1$ is the number of marked points, the second, the
number of edges in the bubble tree. We can derive this formula
\emph{ad hoc}, which will be useful for the subsequent discussion.
We may think of $\calS$ as the moduli space of triples
$\bigl((v^1,z_1),(v^2,z_2),z\bigr)$, where $v^1(z_1)=v^2(z_2)$ and $z$ is the
marked point (different from the nodal point), modulo reparametrisations.
The moduli space of pairs of parametrised spheres representing the class
of $\CP^1$ has dimension $(2+2)n+2c_1(\CP^1)=4n+4$. The marked point $z$
adds $2$ dimensions, the automorphisms fixing $z_1$ and $z_2$, respectively,
have dimension~$4$ for each sphere, and the nodal condition
$v^1(z_1)=v^2(z_2)$ means that the evaluation map $(\ev_1,\ev_2)$
at the nodal points, that is, $\ev_i(v^i)=v^i(z_i)$, $i=1,2$,
maps (transversely) to the diagonal in $\hat{Z}\times\hat{Z}$,
which has codimension~$2n$. In total, this yields
\[ \dim\calS=4n+4+2-2\cdot 4-2n=2n-2,\]
as before.

Now we are in the position to apply what we should like to
call the \emph{moving complex hypersurface argument}.
Observe that the evaluation map
\[ \ev^{\vee}\co\MM^{\vee}\times\CP^1\longrightarrow\hat{Z},\;\;
(u,z)\longmapsto u(z),\]
which maps $z=\infty$ to $H_{\infty}^{\vee}$,
is transverse to the complex hyperplanes $H_0$ and $H_{\infty}$
by Remark~\ref{rem:H-foliate}.

We now appeal to \cite[Theorem~6.6.1]{mcsa12} to see that,
for a residual set of almost complex structures, the evaluation map
$\ev^{\vee}$ defines a $2n$-dimensional pseudocycle. For this we may again
ignore the non-compactness coming from the `ends' of our moduli space,
which are made up entirely of standard spheres. The limit set
of the evaluation map in the sense of \cite[p.~178]{mcsa12}
is contained in $\ev_{\calS}(\calS)$, where $\ev_{\calS}$
denotes the evaluation at the marked point.

By the transversality results for curves with pointwise
constraints in \cite[Section~3.4]{mcsa12},
cf.\ \cite[Section 5.2]{geze16} and \cite[Remark~3.6]{zehm15},
we may assume ---
again for a residual set of almost complex structures ---
that $\ev^{\vee}$ and $\ev_{\calS}$ are
transverse to $\CP^1\times T^*N\times A$, where $A$ is
an arc in $\CP^1$ containing the interval $[0,\infty]\subset
\RP^1$ in its interior.
Then the preimage
\[ \NN:=\bigl(\ev^{\vee}\bigr)^{-1}\bigl(\CP^1\times T^*N\times
[0,\infty]\bigr)\subset\MM^{\vee}\times\CP^1 \]
is a $(2n-1)$-dimensional oriented submanifold of~$\MM^{\vee}\times\CP^1$
with boundary given by the preimage of
$\CP^1\times T^*N\times\{0,\infty\}=
H_0\cup H_{\infty}$; the evaluation map $\ev^{\vee}$ defines a foliation of
these two hyperplanes by standard spheres.
Similarly, the preimage
\[ \NN_{\calS}:=\ev_{\calS}^{-1}\bigl(\CP^1\times T^*N\times
[0,\infty]\bigr)\subset\calS\]
is a submanifold of dimension $2n-3$. It follows that
\[ \ev_{\NN}:=\ev^{\vee}|_{\NN}\co\NN\rightarrow\hat{Z}\]
is a pseudocycle whose limit set is contained in $\ev_{\calS}(\NN_{\calS})$.

In order to arrive at the desired contradiction, we now consider
the intersection of this pseudocycle with the purported bubble sphere $u^2$,
or rather a perturbation of $u^2$ into a smooth (but not
necessarily holomorphic) embedding $w^2\co\CP^1\rightarrow\hat{Z}$.
Because of $u^2\bullet H_0=1$,
the sphere $u^2$ is transverse to~$H_0$ and an embedding near it.
This means that $w^2$ may be chosen to coincide with $u^2$ on
\[ (u^2)^{-1}\bigl(\CP^1\times T^*N\times B_{\varepsilon}\bigr)\]
for some sufficiently small $\varepsilon>0$. Moreover, since 
$u^2\bullet H_{\infty}=0$, we may assume that $w^2$ is likewise
disjoint from~$H_{\infty}$.

Once again we wish to apply the transversality theory for
curves with pointwise constraints, now to $\ev_{\NN}$
and transversality to $w^2(\CP^1)$.
This works fine, as long as we do not run into the
situation that the nodal point of $(v^1,v^2)$ is mapped to~$w^2(\CP^1)$.
This situation can indeed be avoided, as a dimension count shows.
The space of pairs of unparametrised simple spheres
with one marked point each, representing in sum the
class $[\CP^1\times\{v\}\times\{\infty\}]$, has dimension
\[ (2+2)n+2c_1(\CP^1)-2\cdot 4=4n-4.\]
We now look at the evaluation map $(\ev_1,\ev_2)$ at the marked points,
and we make this transverse, by a generic choice
of almost complex structure, to the diagonal $\Delta_{w^2}$ in
$w^2(\CP^1)\times w^2(\CP^1)\subset\hat{Z}\times\hat{Z}$.
Since $\Delta_{w^2}$ has codimension $4n-2$ in $\hat{Z}\times\hat{Z}$,
this means that its preimage under $(\ev_1,\ev_2)$ will be empty, that is,
in the $1$-nodal spheres the nodal point never maps to $w^2(\CP^1)$.

By a further generic choice of almost complex structure we also
achieve transversality to $w^2(\CP^1)$ of the evaluation maps
$\ev_{\NN}$, that is, our pseudocycle, and
$\ev_{\calS}|_{\NN_{\calS}}$, whose image is the limit of the
pseudocycle. Since $w^2(\CP^1)$ has codimension $2n-2$ in~$\hat{Z}$,
this actually implies that the limit set, as the image of
the $(2n-3)$-dimensional manifold $\NN_{\calS}$, is disjoint from
$w^2(\CP^1)$, and $\ev_{\NN}^{-1}\bigl(w^2(\CP^1)\bigr)$ is
a $1$-dimensional submanifold of the $(2n-1)$-dimensional
manifold~$\NN$, in other words,
a finite, disjoint union of properly embedded circles and closed intervals.
Precisely one boundary point of the intervals
belongs to the boundary component $\bigl(\ev^{\vee}\bigr)^{-1}(H_0)$
of~$\NN$, again because
of $u^2\bullet H_0=1$ and Remark~\ref{rem:H-foliate}.
So there would have to be a boundary point in
$\bigl(\ev^{\vee}\bigr)^{-1}(H_{\infty})$.
But this would mean that $w^2\bullet H_{\infty}\neq 0$,
contradicting the fact that $w^2(\CP^1)$
is disjoint from $H_{\infty}$.

So the bubble $u^2$ cannot have existed in the first place.
\end{proof}
\subsection{Infinite holomorphic strips}
\label{subsection:strips}
We continue to assume that hypothesis (H) is satisfied.
In addition, we now assume
that $\pi_1(Q)$ is abelian, so that the homomorphism
$\pi_1(M)\rightarrow\pi_1(W)$ induced by inclusion
is an isomorphism, as explained in Remark~\ref{rem:pi1-abelian}.
This allows us to pass to the universal covers. More generally,
it would suffice to require that the homomorphism
$\pi_1(M)\rightarrow\pi_1(W)$ is injective, since surjectivity
is guaranteed by Theorem~\ref{thm:pi1surjective}.

Recall the definition of $Z^*$ from (\ref{eqn:Z*}).
We consider the proper degree $1$ evaluation map
$\ev\co\MM\times\C^*\rightarrow Z^*$.
Define
\[ Z_{\varepsilon}=\WZ\setminus
\bigl(T^*Q\times B_{\varepsilon}\bigr), \]
i.e.\ $Z_{\varepsilon}$ is obtained from $Z^*$
by removing $T^*Q\times\bigl(B_{\varepsilon}\setminus\{0\}\bigr)$.
Observe that $\partial Z_{\varepsilon}=T^*Q\times\partial D_{\varepsilon}$.

The universal cover of $Z_{\varepsilon}$ is
\[ \wtZ_{\varepsilon} = \wtW\cup_{\wtM}
\bigl( (T^*\wtQ\times[\varepsilon,\infty)\times\R)
\setminus DT^*\wtL\bigr),\]
where $\wtL=\wtQ\times\R$,
and the universal cover of $\partial Z_{\varepsilon}$ equals
\[ \partial\wtZ_{\varepsilon} =
T^*\wtQ\times\{\varepsilon\}\times\R.\]

Set $\C_{\varepsilon}:=\C\setminus B_{\varepsilon}$.
After reparametrising the non-standard discs in $\MM$ so as
to make them look like standard disc on $D_{\varepsilon}\setminus\{0\}$,
we may assume
that the evaluation map restricts to a proper degree~$1$ map
\[ \ev_{\varepsilon}\co(\MM\times\C_{\varepsilon},\MM\times\partial
D_{\varepsilon})\longrightarrow (Z_{\varepsilon},\partial
Z_{\varepsilon}).\]
Notice the slight abuse of notation: elements in $\MM$ are
now understood to be these reparametrised curves.

We now want to construct, on a suitable covering space
$\MM'$ of~$\MM$, a proper degree~$1$ map to
$(Z_{\varepsilon},\partial Z_{\varepsilon})$ that covers
$\ev_{\varepsilon}$. Observe that for every element
$u\co\C_{\varepsilon}\rightarrow Z_{\varepsilon}$ of~$\MM$
we have $u(\varrho)\in H_{\varrho}$ by construction,
where $\varrho$ is the marking from (M2).
Moreover, $u$ and $H_{\varrho}$ intersect uniquely at $\varrho\in\C$.
Therefore, writing
\[ u(\varrho)=:(v,\varrho)\in T^*Q\times\{\varrho\}=
H_{\varrho},\]
we obtain for each lift $\tilde{v}\in T^*\wtQ$ of $v$
a lift
\[ u'\equiv u'_{\tilde{v}}\co [\varepsilon,\infty)\times\R\longrightarrow
\wtZ_{\varepsilon} \]
of $u\in\MM$ with
\[ u'(\varrho,0)=(\tilde{v},\varrho,0)\in
T^*\wtQ\times[\varepsilon,\infty)\times\R.\]
We define $\MM'$ as the set of all such lifts, equipped with
$C^{\infty}_{\loc}$-topology to make it a covering space of $\MM$.
The map $\ev_{\varepsilon}$ lifts to
\[ \ev_{\varepsilon}'\co (\MM'\times[\varepsilon,\infty)\times\R,
  \MM'\times\{\varepsilon\}\times\R) \longrightarrow
(\wtZ_{\varepsilon},\partial\wtZ_{\varepsilon}). \]

\begin{prop}
\label{prop:ev'proper}
The evaluation map $\ev_{\varepsilon}'$ is proper of degree~$1$.
\end{prop}

\begin{proof}
Once properness of $\ev_{\varepsilon}'$ has been established,
the mapping degree is well defined. By looking at standard holomorphic strips
of the form $(r,\theta)\mapsto(\tilde{v},r,\theta)$,
with $\tilde{v}\in T^*\wtQ$ sufficiently large (in fibre
direction), we see that this mapping degree equals~$1$.
 
Regarding properness, we need to show that
the preimage
\[ (\ev_{\varepsilon}')^{-1}(\wtK)\subset
\MM'\times[\varepsilon,\infty)\times\R\]
of any given compact subset $\wtK\subset
\wtZ_{\varepsilon}$ is compact.
Observe that $\wtK$ projects to a compact set
$K\subset Z_{\varepsilon}$, and
\[ (\ev_{\varepsilon})^{-1}(K)\subset\MM\times\C_{\varepsilon}\]
is compact by the properness of~$\ev_{\varepsilon}$.

Let $(u'_{\nu},r_{\nu},\theta_{\nu})$ be a sequence in
$(\ev_{\varepsilon}')^{-1}(\wtK)$.
After selecting a subsequence, we may assume that
$u_{\nu}'(r_{\nu},\theta_{\nu})$
converges to some point $\tilde{p}\in\wtZ_{\varepsilon}$,
and that the projection
\[ (u_{\nu},z_{\nu}=r_{\nu}\rme^{\rmi\theta_{\nu}})
\in(\ev_{\varepsilon})^{-1}(K)\subset\MM\times\C_{\varepsilon} \]
of $(u'_{\nu},r_{\nu},\theta_{\nu})$ converges to a pair $(u_0,z_0)$.
In particular, writing $z_0=r_0\rme^{\rmi\theta_0}$,
this means that $\theta_{\nu}$ converges to
$\theta_0$ modulo $2\pi$, and $r_{\nu}$ converges to $r_0$.
The point $\tilde{p}$ projects to $p=u_0(z_0)$. The situation is
summarised in the following diagram:

\begin{diagram}
(u_{\nu}',r_{\nu},\theta_{\nu})\rightarrow(?,r_0,\theta_0\,\text{mod $2\pi$})
\in && \MM'\times\widetilde{\C}_{\varepsilon}
    && \rTo^{\ev_{\varepsilon}'}
    & \wtZ_{\varepsilon} & \supset\wtK\ni\tilde{p}\\
    && \dTo &&  & \dTo &            & \\
(u_{\nu},r_{\nu}\rme^{i\theta_{\nu}})\rightarrow (u_0,z_0)
\in && \MM\times\C_{\varepsilon}
    && \rTo^{\ev_{\varepsilon}}
    & Z_{\varepsilon}             & \supset K\ni p.
\end{diagram}

Choose $k_{\nu}\in\Z$ such that
\[ \theta_{\nu}-2\pi k_{\nu}\longrightarrow \theta_0\;\;\;
\text{for}\;\;\nu\rightarrow\infty.\]
Write $\ell_{\nu}$ for the line segment in $[\varepsilon,\infty)\times\R$
connecting $(\varrho,2\pi k_{\nu})$ with $(r_{\nu},\theta_{\nu})$.
The arc $u_{\nu}'(\ell_{\nu})$ projects to $u_{\nu}(\ell_{\nu})$.
Since $u_{\nu}\rightarrow u_0$ in $\MM$, the distance
\[ d\bigl(u_{\nu}'(\varrho,2\pi k_{\nu}),
u_{\nu}'(r_{\nu},\theta_{\nu})\bigr)\]
in $\wtZ_{\varepsilon}$ is uniformly bounded from above
by a constant times the length of~$\ell_{\nu}$, and hence stays bounded as
$\nu\rightarrow\infty$. Recall that $u_{\nu}'(\varrho,2\pi k_{\nu})$
may be written as $(\tilde{v}_{\nu},\varrho,2\pi k_{\nu})$.
Then, from $u_{\nu}'(r_{\nu},\theta_{\nu})\rightarrow\tilde{p}$,
we deduce that $(\tilde{v}_{\nu},\varrho,2\pi k_{\nu})$ stays
at bounded distance from $\tilde{p}$ in~$\wtZ_{\varepsilon}$.

Hence, after passing to a subsequence we may assume that
the sequence $(\tilde{v}_{\nu})$ converges to some point $\tilde{v}_0
\in T^*\wtQ$, and $k_{\nu}=k_0$ for all~$\nu$. Then
the sequence $(u_{\nu}') $ of lifted holomorphic curves converges
to $u_0'\in\widetilde{\MM}$, characterised by
$u_0'(\varrho,0)=(\tilde{v}_0,\varrho,0)$.
\end{proof}
\section{Proof of the main result}
We are now in the position to prove our main result Theorem~\ref{thm:main}
about the uniqueness of certain fillings up to homotopy equivalence
or up to diffeomorphism, respectively.
\subsection{Homotopy type}
\label{subsection:homotopy-type}
In this section we prove part (a) of Theorem~\ref{thm:main}.

\begin{prop}
\label{prop:homsurj}
Let $Q$ be a closed, connected manifold and
$(W,\omega)$ a symplectically aspherical filling of $M=ST^*(Q\times S^1)$
satisfying hypothesis~{\rm (H)}.
Assume that $\pi_1(Q)$ is abelian or, more generally,
that the inclusion $M\rightarrow W$ is $\pi_1$-injective.
Then the composition
$\wtQ\rightarrow\wtM\rightarrow\wtW$
of inclusion maps of universal covers
is surjective on homology.
\end{prop}

\begin{proof}
The evaluation map $\ev_{\varepsilon}'$ from Proposition~\ref{prop:ev'proper}
fits into the following diagram, where the horizontal map
at the top is the obvious restriction of~$\ev_{\varepsilon}'$,
and the vertical maps are inclusions:

\begin{diagram}
\MM'\times\{\varepsilon\}\times\R      && \rTo^{\ev_{\varepsilon}'} &&
   \partial\wtZ_{\varepsilon}\\
\dTo^{\simeq}                          &&                           &&
   \dTo\\
\MM'\times[\varepsilon,\infty)\times\R && \rTo^{\ev_{\varepsilon}'} &&
   \wtZ_{\varepsilon}.
\end{diagram}

A homological argument as in \cite[Proposition~2.4]{bgz}
shows that $\ev_{\varepsilon}'$ at the bottom is surjective
on homology. It follows that the inclusion map
$\partial\wtZ_{\varepsilon}\rightarrow\wtZ_{\varepsilon}$
must likewise be surjective on homology. Observe that
$\partial\wtZ_{\varepsilon}=T^*\wtQ\times
\{\varepsilon\}\times\R$ strongly deformation
retracts onto $\wtQ\times\{\varepsilon\}\times\{0\}$,
where $\wtQ$ is regarded as the zero section of~$T^*\wtQ$.
Also, $\wtZ_{\varepsilon}$ strongly deformation retracts
onto~$\wtW$. The latter retraction may be assumed to
send $\wtQ\times\{\varepsilon\}\times\{0\}$ to the
copy of $\wtQ$ in $\wtM$ coming from the
inclusion $Q\times S^1\subset M$ described in
Section~\ref{subsection:split}. This means that the inclusion
$\partial\wtZ_{\varepsilon}\rightarrow\wtZ_{\varepsilon}$
retracts to the composition
$\wtQ\rightarrow\wtM\rightarrow\wtW$
of inclusions. This proves the proposition.
\end{proof}

Recall from Section~\ref{subsection:split}
the definition of the copy $W_0\subset W$ of $DT^*(Q\times S^1)$.
This $W_0$ retracts to $Q\times S^1\subset M\subset W$,
so the inclusion $\wtW_0\rightarrow\wtW$
is likewise surjective on homology.

The following theorem is the analogue of results in
\cite[Section~7]{bgz} in the present more general setting.

\begin{thm}
\label{thm:homotopequiv}
Let $Q$ be a closed, connected manifold and $(W,\omega)$
a symplectically aspherical filling of $M=ST^*(Q\times S^1)$
satisfying hypothesis~{\rm (H)}.
Assume further that one of the following conditions is satisfied:
\begin{itemize}
\item[($\alpha$)] $Q$ is aspherical and $\pi_1(Q)$ abelian;
\item[($\beta$)] $M$ is a simple space.
\end{itemize}
Then $W$ has the homotopy type of $DT^*(Q\times S^1)$.
\end{thm}

Since the fundamental group of an aspherical manifold is torsion-free,
and the topological Borel conjecture holds for the
torus~\cite{luck10}, ($\alpha$)
implies that $Q$ is homeomorphic to a torus. However,
as in Proposition~\ref{prop:homsurj}, we may replace the requirement
in ($\alpha$) that $\pi_1(Q)$ be abelian by the less restrictive
assumption that the inclusion $M\rightarrow W$ is $\pi_1$-injective.

Recall that a topological space is called \emph{simple} if
its fundamental group acts trivially on all its homotopy group.
The action of $\pi_1$ on itself is given by conjugation, so a simple
space has an abelian fundamental group.

\begin{proof}[Proof of Theorem~\ref{thm:homotopequiv}]
Under assumption ($\alpha$), we argue as follows.
The universal cover $\wtQ$ is contractible,
hence so is $\wtW$ by Proposition~\ref{prop:homsurj}.
So the inclusion $\wtW_0\rightarrow\wtW$
of universal covers is a homotopy equivalence
(of contractible spaces). This means that $H_k(\wtW,\wtW_0)=0$
for all~$k$.

The remaining argument is analogous
to the proof of \cite[Theorem~7.2]{bgz}. Here are the details.
By excision, and with the cobordism $\{M_0,X,M\}$ as defined
in Section~\ref{subsection:cobordism}, we have $H_k(\wtX,\wtM_0)=0$
for all~$k$. The relative Hurewicz theorem implies $\pi_k(\wtX,\wtM_0)=0$
for all $k$. Since the inclusion $M_0\rightarrow X$ is an isomorphism
on~$\pi_1$, as remarked after the proof of Lemma~\ref{lem:cobordism},
we also have $\pi_k(X,M_0)=0$ for all~$k$. Whitehead's theorem then
implies that $M_0$ is a strong deformation retract of~$X$.
Hence $W=W_0\cup_{M_0}X\simeq W_0$.
 
Under assumption ($\beta$), the fundamental group $\pi_1(Q)$
must be abelian, since the fundamental groups
of $M=ST^*(Q\times S^1)$ and $Q\times S^1$ are isomorphic,
and the former is abelian thanks to $M$ being a simple space.

As before, we need to show that all relative homotopy groups $\pi_k(X,M_0)$
are zero. This is precisely the content of \cite[Lemma~8.1]{bgz};
we only need to verify that the subsidiary results cited in the
proof of that lemma are available in our more general setting
discussed here.

First of all, we need the inclusion $M_0\rightarrow X$ to be
$\pi_1$-isomorphic, which is guaranteed by Theorem~\ref{thm:pi1surjective},
Remark~\ref{rem:pi1-abelian} and the comments at the end of
Section~\ref{subsection:cobordism}. Secondly, we need the
inclusion $\wtM_0\rightarrow\wtW$ to be surjective on homology,
which is guaranteed by Proposition~\ref{prop:homsurj}.
\end{proof}

For Stein fillings $(W,\omega)$ of dimension at least $6$
we see from the handlebody structure
that $H_2(W,M)=0$ and $\pi_2(W,M)=0$ by general position.
Hence, the inclusion $M\rightarrow W$ is $\pi_1$-injective,
and thus $\pi_1$-isomorphic by Theorem~\ref{thm:pi1surjective}.
This observation leads to the following theorem.

\begin{thm}
\label{thm:Stein}
Let $Q$ be a closed, connected manifold of dimension $n-1\geq 2$
with $\pi_1(Q)$ infinite and $H_{n-2}(\widetilde{Q})=0$.
Let $(W,\omega)$ be a Stein filling of $M=ST^*(Q\times S^1)$.
Assume that $\chi(Q)=0$, so that hypothesis {\rm (H-i)}
is satisfied. Then $W$ is homotopy equivalent to $DT^*(Q\times S^1)$.
\end{thm}
 
\begin{proof}
The $2n$-dimensional Stein manifold $W$ has the homotopy type of a cell
complex of dimension at most~$n$. In fact, there are no
subcritical fillings of unit cotangent bundles,
see~\cite[Proposition~3.9]{bgz}, so the cellular dimension
of $W$ is actually equal to~$n$. From
\[ H_k(W,M)\cong H^{2n-k}(W)\cong FH_{2n-k}(W)\oplus
TH_{2n-k-1}(W)\]
we conclude that $H_k(W,M)=0$ for $k\leq n-1$, since
$FH_{\ell}(W)=0$ for $\ell\geq n+1$, and $TH_{\ell}=0$ for
$\ell\geq n$, as there are no $(n+1)$-cells in~$W$.

Similarly one sees that $\pi_k(W,M)=0$ for $k\leq n-1$,
since any relative $k$-disc with $k\leq n-1$ can be made disjoint
from the $n$-dimensional cellular skeleton of~$W$.
It follows that $\pi_k(\wtW,\wtM)=0$ and $H_k(\wtW,\wtM)=0$
for $k\leq n-1$.

As earlier, we regard $L=Q\times S^1$ as a section of $M=ST^*(Q\times S^1)$.
By general position, we have $\pi_k(M,L)=0$
for $k\leq n-2$, since in this range a relative $k$-disc can be made
disjoint from the section antipodal to~$L$, and then be pushed into~$L$.
It then follows that $H_k(\wtM,\wtL)=0$ for $k\leq n-2$.

From the homology long exact sequence of the triple
$(\wtW,\wtM,\wtL)$ we then deduce that
$H_k(\wtW,\wtW_0)=H_k(\wtW,\wtL)=0$ for $k\leq n-2$. So the
inclusion $\wtW_0\rightarrow\wtW$ induces isomorphisms on homology
in degrees $k\leq n-3$.

On the other hand, the assumption on $\pi_1(Q)$ being infinite implies
that $\wtQ$ is not compact, and hence $H_{n-1}(\wtQ)=0$.
With the further homological assumption on $\wtQ$ in the theorem,
we have $H_k(\wtQ)=0$ for $k\geq n-2$. With Proposition~\ref{prop:homsurj}
we find that $H_k(\wtW)=0$ for $k\geq n-2$. The same is obviously
true for $\wtW_0$.

Thus, the inclusion $\wtW_0\rightarrow\wtW$ induces an isomorphism
on all homology groups, and hence $H_k(\wtW,\wtW_0)=0$ for all~$k$.
The argument now concludes as in case ($\alpha$) of
Theorem~\ref{thm:homotopequiv}.
\end{proof}

\begin{proof}[Proof of Theorem~\ref{thm:main}~(a)]
The manifolds in (a-i) and (a-ii) satisfy the assumptions
of Theorem~\ref{thm:Stein}.
\end{proof}
\subsection{Diffeomorphism type}
\label{subsection:difftype}
In this section we prove part (b) of Theorem~\ref{thm:main}.
The key point is to show that the cobordism $\{M_0,X,M\}$,
under appropriate assumptions, is an $h$-cobordism. Compared with the
previous discussion, one needs to ensure that the `upper' inclusion
$M\rightarrow X$ is likewise a homotopy equivalence. An additional
assumption on the vanishing of the Whitehead group $\Wh(\pi_1(M))$
then guarantees $\{M_0,X,M\}$ to be an $s$-cobordism, and hence
diffeomorphic to a product $M\times [0,1]$ by the $s$-cobordism theorem.

\begin{thm}
\label{thm:diff-type}
Let $Q$ be a closed, connected manifold
and $(W,\omega)$ a filling
of $M=ST^*(Q\times S^1)$ satisfying hypothesis~{\rm (H)}.
If $M$ is a simple space and $\Wh(\pi_1(M))=0$, then
$W$ is diffeomorphic to $DT^*(Q\times S^1)$.
\end{thm}

\begin{proof}
The `lower' inclusion $M_0\rightarrow X$ is a homotopy equivalence
by the proof of Theorem~\ref{thm:homotopequiv} under assumption~($\beta$).
The `upper' inclusion $M\rightarrow X$ is a homotopy equivalence
thanks to \cite[Lemma 8.2]{bgz}, which only uses the simplicity
of~$M_0$ (and hence that of~$X$).
\end{proof}

\begin{rem}
Under assumption ($\alpha$) of Theorem~\ref{thm:homotopequiv},
the above argument breaks down, because the corresponding
argument in \cite{bgz} is based on Lemma~5.2
in that paper, for which the assumption that the filling is
subcritical is crucial.
\end{rem}

\begin{proof}[Proof of Theorem~\ref{thm:main}~(b)]
The manifolds $Q$ in Theorem~\ref{thm:main}~(b) satisfy
hypothesis (H-ii). The remaining assumptions of
Theorem~\ref{thm:diff-type} are satisfied
thanks to \cite[Example~9.3.(2)]{bgz}.
\end{proof}
\begin{ack}
A part of this research was done while all three authors
took part in the workshop \emph{Conservative dynamics and its
interactions} at the Bernoulli Center (CIB) of the EPFL,
organised by Felix Schlenk and Leonid Polterovich.
We are grateful to the organisers and the Bernoulli Center
for providing ideal conditions for collaborative research.
We thank Jonathan Bowden for useful comments on a draft
version of this paper, and Thomas Schick for the
reference to the Borel conjecture. We thank an anonymous referee
for drawing our attention to an incomplete argument in a previous
version of the proof of Lemma~\ref{lem:bubbling-ii}.
\end{ack}

\end{document}